\documentclass{amsart}
\usepackage{amssymb}
\usepackage{hyperref}
\usepackage{url}
\usepackage[all]{xy}
\usepackage{tikz}
\usepackage[none]{hyphenat}

\usetikzlibrary{matrix,arrows,decorations.pathmorphing}

\newtheorem{thm}{Theorem}
\newtheorem{introthm}{Theorem}
\newtheorem{introcor}{Corollary}

\newtheorem{lem}[thm]{Lemma}

\newtheorem{prop}[thm]{Proposition}
\newtheorem{introcon}{Conjecture}

\theoremstyle{definition}
\newtheorem{defn}[thm]{Definition}
\newtheorem{introdefn}{Definition}
\newtheorem{rem}[thm]{Remark}

\newcommand{\RN}[1]{%
  \textup{\uppercase\expandafter{\romannumeral#1}}%
}

\numberwithin{thm}{section}

\def\Z{\mathbf{Z}}
\def\Q{\mathbf{Q}}
\def\F{\mathbf{F}}

\def\Fp{\F_p}
\def\Ftwo{\F_2}

\def\bA{\mathbf{A}}

\def\Xset{\mathcal{C}}

\def\Xset{\mathcal{C}}

\def\A{\mathcal{A}}
\def\O{\mathcal{O}}

\def\cP{\mathcal{P}}

\def\cP{\mathcal{P}}
\def\H{\mathcal{H}}

    \DeclareFontFamily{U}{wncy}{}
    \DeclareFontShape{U}{wncy}{m}{n}{<->wncyr10}{}
    \DeclareSymbolFont{mcy}{U}{wncy}{m}{n}
    \DeclareMathSymbol{\Sh}{\mathord}{mcy}{"58}

\def\l{\mathfrak{q}}

\def\p{\mathfrak{p}}
\def\q{\mathfrak{q}}

\def\d{\mathfrak{d}}

\def\Gl{\mathrm{GL}}

\def\Hom{\mathrm{Hom}}
\def\Gal{\mathrm{Gal}}

\def\Sel{\mathrm{Sel}}

\def\Frob{\mathrm{Frob}}

\def\ur{\mathrm{ur}}

\def\res{\mathrm{res}}

\def\dimp{\dim_{\Fp}}
\def\dimtwo{\dim_{\Ftwo}}

\def\ram{\mathrm{ram}}

\def\Pic{\mathrm{Pic}}
\def\iK{\bA_K^\times}

\def\N{\mathbf{N}}

\def\too{\longrightarrow}

\def\dirsum#1{\underset{#1}{\textstyle\bigoplus}}

\def\d2{\dim_{\Ftwo}}

\title[$2$-Selmer near-companion curves]
    {$2$-Selmer near-companion curves}
\author{Myungjun Yu}
\address{Department of Mathematics,
University of Michigan, Ann Arbor, MI 48109-1043,
USA}

\email{\href{mailto:myungjuy@umich.edu}{myungjuy@umich.edu}}

\begin{document}

\begin{abstract}
Let $E$ and $A$ be elliptic curves over a number field $K$. Let $\chi$ be a quadratic character of $K$. We prove the conjecture posed by Mazur and Rubin on $n$-Selmer near-companion curves in the case $n=2$. Namely, we show if the difference of the $2$-Selmer ranks of $E^\chi$ and $A^\chi$ is bounded independent of $\chi$, there is a $G_K$-module isomorphism $E[2] \cong A[2]$. 
\end{abstract}

\maketitle
\sloppy

\section*{Introduction}
In \cite{companion}, Mazur and Rubin study the so called Selmer companion curves. By definition, elliptic curves $E$ and $A$ over a number field $K$ are called $n$-Selmer companion curves over $K$ if for every quadratic character $\chi \in \Hom(G_K, \{\pm1\})$, where $G_K$ is the absolute Galois group of $K$, there exists a group isomorphism between $\Sel_n(E^\chi/K)$ and $\Sel_n(A^\chi/K)$, where $E^\chi$ and $A^\chi$ denote the quadratic twists of $E$ and $A$ by $\chi$, respectively. This definition originates from the question that asks what information about an elliptic curve $E/K$ could be read off from the function 
$$\Hom(G_K, \{\pm1\}) \to \N_{\ge 0}$$ 
taking $\chi$ to $\dim_{\Fp}(\Sel_p(E^\chi/K))$ for a prime $p$. They also defined a weaker condition on $E$ and $A$ the so called $n$-Selmer near-companion curves (\cite[Definition 7.12]{companion}) over $K$. To simplify the definition, we define $n$-Selmer near-companion curves only when $n$ is a prime number, which is sufficient for our purposes.

\begin{introdefn}
Let $p$ be a prime number. Let $E$ and $A$ be elliptic curves over a number field $K$. We say $E$ and $A$ are {\em $p$-Selmer near-companion curves} over $K$ if there exists a constant $C := C(E,A,K)$ such that for every $\chi \in \Hom(G_K, \{\pm 1\})$, 
$$|\dimp(\Sel_p(E^\chi/K)) - \dimp(\Sel_p(A^\chi/K))| < C.$$
\end{introdefn}

If elliptic curves $E$ and $A$ over a number field $K$ are randomly chosen, there is no reason to expect they are $p$-Selmer near-companion curves over $K$. Therefore, it seems natural to expect that if $E$ and $A$ are $p$-Selmer near-companion curves over $K$, they should be closely related. In fact, Mazur and Rubin conjectured the following (\cite[Conjecture 7.15]{companion}).

\begin{introcon}
Suppose that elliptic curves $E$ and $A$ over a number field $K$ are $p$-Selmer near-companion curves over $K$, then there exists a $G_K$-module isomorphism $E[p] \cong A[p]$. 
\end{introcon}
They proved the converse (with a stronger assumption if $p =2 \text{ or } 3$) of the conjecture as follows (\cite[Theorem 7.13]{companion}).
\begin{introthm}[Mazur-Rubin]
\label{int1}
Let $E$ and $A$ be elliptic curves over a number field $K$. Let $m_p = p^2$ if $p =2 \text{ or } 3$, and $m_p= p$ if $p>3$. Suppose that there is a $G_K$-module isomorphism $E[m_p] \cong A[m_p]$. Then $E$ and $A$ are $p$-Selmer near-companion curves over $K$.
\end{introthm}

Let us brifely discuss the idea of the proof of Theorem \ref{int1}. Suppose that there is a $G_K$-module isomorphism $E[m_p] \cong A[m_p]$. Then one can prove that for every $\chi \in \Hom(G_K, \{\pm 1\})$, the local conditions of $\Sel_p(E^\chi/K)$ and $\Sel_p(A^\chi/K)$ are the same everywhere except possibly at places in some finite set $S_{(E,K,p)}$ of places of $K$, which is independent of the choice of $\chi$. Applying \cite[Theorem 2.3.4]{kolyvagin}, one can prove
$$
|\dimp(\Sel_p(E^\chi/K))-\dimp(\Sel_p(A^\chi/K))| \le \displaystyle\sum_{v \in S_{(E,K,p)}} \dimp(E^\chi(K_v)/pE^\chi(K_v)).
$$
Moreover the RHS of the inequality is bounded by a certain constant, which is independent of the choice of $\chi$. For example, when $v$ is a prime of $K$ above $p$, we have $$\dimp(E^\chi(K_v)/pE^\chi(K_v)) =  [K_v:\Q_p]  + \dimp(E^\chi(K_v)[p]) \le [K_v:\Q_p] + 2.$$

The main theorem of this article is the following.
\begin{introthm}
\label{maintheorem}
Suppose that elliptic curves $E$ and $A$ over a number field $K$ are $2$-Selmer near-companion curves over $K$, then there is a $G_K$-module isomorphism $E[2] \cong A[2]$.
\end{introthm}

We prove its contrapositive. First note that $K(E[2]) = K(A[2])$ if and only if there is a $G_K$-module isomorphism $E[2] \cong A[2]$, that is Lemma \ref{equivcond}. We prove the following theorems by applying Theorem \ref{case1}, Theorem \ref{case2}, Theorem \ref{case3} inductively (the assumptions on these theorems are only about the extensions $K(E[2])/K$ and $K(A[2])/K$, and they are invariants under quadratic twists by Remark \ref{canisom}, so we can use induction).  

\begin{introthm}
\label{1}
Suppose that $K(E[2]) \neq K(A[2])$ and $[K(E[2]):K] \le [K(A[2]):K]$. Then for any integer $d$, there exist infinitely many $\chi \in \Hom(G_K, \{\pm1\})$ such that
$$\dimtwo(\Sel_2(E^\chi/K)) - \dimtwo(\Sel_2(A^\chi/K))  > d.$$
\end{introthm}

\begin{introthm}
\label{2}
Suppose that $K(E[2]) \neq K(A[2])$. Suppose that one of the following conditions holds.
\begin{itemize}
\item
$[K(E[2]):K]$ and $[K(A[2]):K]$ are divisible by $3$, or
\item
$[K(E[2]):K] = [K(A[2]):K] = 2$.
\end{itemize}
Then for any integer $d$, there exist infinitely many $\chi_1, \chi_2  \in \Hom(G_K, \{\pm1\})$ such that
\begin{enumerate}
\item
$\dimtwo(\Sel_2(E^{\chi_1}/K)) - \dimtwo(\Sel_2(A^{\chi_1}/K))  > d$, and
\item
$\dimtwo(\Sel_2(A^{\chi_2}/K)) - \dimtwo(\Sel_2(E^{\chi_2}/K))  > d$.
\end{enumerate}
\end{introthm}
Let $C$ be an elliptic curve over $K$ and let $r_{C/K}$ denote the Mordell-Weil rank of $C$ over $K$. Then the following exact sequence
$$
0 \to C(K)/2C(K) \to \Sel_2(C/K) \to \Sh_{C/K}[2] \to 0,
$$
where $\Sh_{C/K}$ denotes the Shafarevich-Tate group of $C$ over $K$, together with the above theorems imply the following Corollaries. 

\begin{introcor}
Suppose that the hypotheses in Theorem \ref{1} hold. Then for any integer $d$, there exist infinitely many $\chi \in \Hom(G_K, \{\pm1\})$ such that
$$r_{E^\chi/K} - r_{A^\chi/K} > d, \text{ or } \dimtwo(\Sh_{E^\chi/K}[2]) - \dimtwo(\Sh_{A^\chi/K}[2]) > d.$$
\end{introcor}

\begin{introcor}
Suppose that the hypotheses in Theorem \ref{2} hold. Then for any integer $d$, there exist infinitely many $\chi_1, \chi_2  \in \Hom(G_K, \{\pm1\})$ such that
\begin{enumerate}
\item
$r_{E^{\chi_1}/K} - r_{A^{\chi_1}/K} > d, \text{ or } \dimtwo(\Sh_{E^{\chi_1}/K}[2]) - \dimtwo(\Sh_{A^{\chi_1}/K}[2]) > d$, 
\item
$r_{A^{\chi_2}/K} - r_{E^{\chi_2}/K} > d, \text{ or } \dimtwo(\Sh_{A^{\chi_2}/K}[2]) - \dimtwo(\Sh_{E^{\chi_2}/K}[2]) > d$.
\end{enumerate}
\end{introcor}

When $E$ and $A$ are elliptic curves over a number field $K$, and are isogeneous over $K$ of degree $n$, where $n$ is coprime to $2$,  it is easy to see that $E$ and $A$ are $2$-Selmer companion curves over $K$ since the isogeny induces the isomorphism $\Sel_2(E^\chi/K) \cong \Sel_2(A^\chi/K)$ for any $\chi \in \Hom(G_K, \{\pm1\})$. It is interesting to note that when $E$ and $A$ are isogenous over $K$ of degree $2$, it is possible that $E$ and $A$ are not even $2$-Selmer near-companion curves over $K$. A large family of exampes of such $E$ and $A$ is given by Theorem \ref{1}. 

\begin{introcor}
Let $E$ and $A$ be elliptic curves over a number field $K$. Suppose that $E$ and $A$ are isogenous over $K$ of degree $2$. Suppose further $K(E[2]) \neq K(A[2])$, then $E$ and $A$ are not $2$-Selmer near-companion curves over $K$.   
\end{introcor}

Due to lack of the tools to analyze the behavior of the $p$-Selmer ranks in the family of quadratic twists for prime numbers $p$ other than $2$, we focus on the case $p=2$ in the present article. The point is that when $p=2$, for any $\chi \in \Hom(G_K, \{\pm 1\})$, there is a canonical $G_K$-module isomorphism between $E[2]$ and $E^\chi[2]$, so we can compare $\Sel_2(E/K)$ and $\Sel_2(E^\chi/K)$ in the same cohomology group $H^1(K, E[2])$. Moreover, local conditions of $\Sel_2(E/K)$ and $\Sel_2(E^\chi/K)$ at a place $v$ are in the same cohomology group $H^1(K_v, E[2])$, so we can compare the local conditions. In fact, the main idea is to choose $\chi$ so that the local conditions of $\Sel_2(E/K)$ and $\Sel_2(E^\chi/K)$ are the same everywhere except one place, where we get a rank variation in a controllable way.

However for an odd prime $p$, there is no canonical $G_K$-module isomorphism between $E[p]$ and $E^\chi[p]$, so it is rather difficult to compute $\dimp(\Sel_p(E^\chi/K))$ from the information of  $\dimp(\Sel_p(E/K))$ (local conditions of $\Sel_p(E/K)$ and $\Sel_p(E^\chi/K)$ are not comparable).

We write $M, M'$ for $K(E[2]), K(A[2])$, respectively. In the rest of the paper, we suppose $M \neq M'$ for our purpose. In section $3$, we define metabolic spaces and Lagrangian subspaces. With these notions, we prove Proposition \ref{crucial}, which plays an important role to show our main theorem. We prove the main theorem by case-by-case. In Section $4$, we deal with the case $[M:K], [M':K]$ are divisible by $3$. For the rest of cases, we assume $[M:K] \le [M':K]$ by symmetry. In Section $5$, we take care of the case $[M:K] = 1$ or $2$, and $[M':K]$ is divisible by 3. Section $6$ will be devoted to the last case: $[M:K] = 1$ or $2$, and $[M':K] = 2$.


\section{Preliminaries}
We fix a number field $K$.  We write $K_v$ for the completion of $K$ at a place $v$. In addition, we fix an embedding $\overline{K} \to \overline{K_v}$ for every place $v$ of $K$, so that $G_{K_v} \subset G_K$, where $G_{K_v}$, $G_K$ denote the absolute Galois groups of $K_v$, $K$, respectively. In this section, let $C$ denote an elliptic curve over $K$.

\begin{defn}
Let $L$ be a field of characteristic $0$. We write
$$
\Xset(L) := \Hom(G_L, \{\pm1\}).
$$
If $L$ is a local field, we often identify $\Xset(L)$ with $\Hom(L^\times, \{\pm1\})$ via the local reciprocity map, and let $\Xset_\ram(L) \subset \Xset(L)$ be the subset of ramified characters in $\Xset(L)$ (by local class field theory, $\chi \in \Xset_\ram(L)$ if and only if $\chi(\O_L^\times) \neq 1$, where $\O_L^\times$ is the unit group of $L$). 
\end{defn}

\begin{rem}
\label{canisom}
For any $\chi \in \Xset(K)$, note that there is a canonical $G_K$-module isomorphism
$$
C[2] \cong C^\chi[2],
$$
which is the restriction of the canonical isomorphism $\phi: C \to C^\chi$, where $\phi^\sigma(P) = \chi(\sigma)\phi(P)$ for $P \in C(\overline{K})$ and $\sigma \in G_K$. Indeed, if $P \in C[2]$, then for any $\sigma \in G_K$
$$
\phi^\sigma(P) = \chi(\sigma)\phi(P) = \pm \phi(P) = \phi(P).
$$
The same is true if $K$ is replaced by $K_v$. 
\end{rem}

\begin{lem}
\label{locdim2torsion}
Let $\l\nmid 2$ be a prime of $K$. Then 
$$
\dimtwo(C(K_\l)/2C(K_\l)) = \dimtwo(C(K_\l)[2]).
$$
\end{lem}

\begin{proof}
For any odd prime $p$, the pro-$p$ part of $C(K_\l)$ is $2$-divisible. Therefore, there is an isomorphism 
$$
C(K_\l)/2C(K_\l) \cong  C(K_\l)[2^\infty]/2C(K_\l)[2^\infty],
$$
so the lemma follows.
\end{proof}

\begin{thm}
\label{tlp}
The Tate local duality and the Weil pairing induce a nondegenerate pairing
\begin{equation}
\label{tld}
\langle  \text{ , } \rangle_v : H^1(K_v, C[2]) \times H^1(K_v, C[2]) \too H^2(K_v, \{\pm1\}),
\end{equation}
where $H^2(K_v, \{\pm1\}) \cong \Ftwo$ unless $v$ is a complex place.
\end{thm}

\begin{proof}
For example, see \cite[Theorem 7.2.6]{cohomology}.
\end{proof}

\begin{lem}
\label{equivcond}
Let $E, A$ be elliptic curves over $K$. Then $K(E[2]) = K(A[2])$ if and only if there exists a $G_K$-module isomorphism $E[2] \cong A[2]$
\end{lem}

\begin{proof}
If there is a $G_K$-module isomorphism $E[2] \cong A[2]$, then
$$
\{ \sigma \in G_K : \sigma P = P \text{ for every } P \in E[2]\} = \{ \sigma \in G_K : \sigma P' = P' \text{ for every } P' \in A[2]\},
$$
so $K(E[2]) = K(A[2])$ (the fixed field of the group above). Now we prove the converse. Fixing bases of $E[2], A[2]$ over $\Ftwo$, one can define injective maps $\phi : \Gal(K(E[2])/K) \hookrightarrow \Gl_2(\Ftwo)$ and $\phi' : \Gal(K(A[2])/K) \hookrightarrow \Gl_2(\Ftwo)$. Elementary group theory proves that any (group) isomorphism between two (possibly the same) subgroups of  $\Gl_2(\Ftwo)$ is given by the restriction of an inner automorphism of  $\Gl_2(\Ftwo)$ ($\cong S_3$). Therefore after an appropriate change-of-basis, we can identify the maps $\phi$ and  $\phi'$, where we derive a $G_K$-module isomorphism $E[2] \cong A[2]$.
\end{proof}

The following lemmas will be used frequently in later sections. 

\begin{lem}
\label{cft}
Let $S$ be a (finite) set of places of $K$ such that $\Pic(\O_{K,S}) =1$, where $\O_{K, S}$ denotes the ring of $S$-integers of $K$. The image of the restriction map
\begin{align*}
\Xset(K) = \Hom(\iK/K^\times,\{\pm1\}) &= \textstyle\Hom((\prod_{\mu \in S}K_{\mu}^\times \times \prod_{\nu\notin S}\O_{\nu}^\times)/\O_{K, S}^\times,\{\pm1\}) \\
& \too \textstyle{\prod_{\mu \in S}\Hom(K_{\mu}^\times, \{\pm1\})\times \prod_{\nu \not\in S}\Hom(\O_{\nu}^\times, \{\pm1\})}
\end{align*}
is the set of all $((f_{\mu})_{\mu \in S}, (g_{\nu})_{\nu \not\in S})$ such that $\prod_{\mu \in S} f_{\mu}(b) \prod_{\nu \not\in S} g_{\nu}(b) = 1$ for all $b \in \O_{K, S}^\times$, where $f_{\mu} \in \Hom(K_{\mu}^\times, \{\pm1\})$ and $g_{\nu}\in \Hom(\O_{\nu}^\times, \{\pm1\})$.
\end{lem}

\begin{proof}
The Lemma follows from the Class Field Theory.
\end{proof}

\begin{lem}
\label{galois}
Let $\mathcal{K}$ be a field of characteristic not equal to $2$. Suppose $\mathcal{M}$ is a Galois extension of $\mathcal{K}$. Suppose that $\mathcal{N} \subset \mathcal{M}(\sqrt{\mathcal{M^\times}})$ is an extension of $\mathcal{M}$,i.e., $\mathcal{N}$ is a compositum of quadratic extensisons of $\mathcal{M}$. Then the Galois closure of $\mathcal{N}$ over $\mathcal{K}$ is contained in $\mathcal{M}(\sqrt{\mathcal{M^\times}})$.
\end{lem}

\begin{proof}
The Galois closure is a compositum of $\sigma \mathcal{N}$, where $\sigma$ is an embedding from $\mathcal{N}$ to $\overline{\mathcal{K}}$. Since $\sigma(\mathcal{M}) = M$, and $\mathcal{N}$ is a compositum of quadratic extensions of $\mathcal{M}$, it follows that $\sigma \mathcal{N}$ is also a compositum of quadratic extensions of $\mathcal{M}$. 
\end{proof}

\begin{lem}
\label{boolean} 
The Galois group $\Gal(K(C[4])/K(C[2]))$ is a Boolean group, i.e., an abelian group in which every non trivial element has order $2$. In other words, $K(C[4])$ is a compositum of quadratic extensions of $K(C[2])$. 
\end{lem}

\begin{proof}
For any $\sigma \in \Gal(K(C[4])/K(C[2]))$ and $P \in C[4]$, since $\sigma(2P) = 2P$, we have $\sigma(P) - P \in C[2]$, so $\sigma(\sigma(P) - P) + (\sigma(P) - P) = 0$, i.e., $\sigma^2(P) = P$.  
\end{proof}


\section{Selmer groups and local conditions}
From now on, fix elliptic curves $E$ and $A$ over $K$. For the rest of the paper, let $M:= K(E[2])$ and $M':= K(A[2])$. For our main purpose, we suppose $M \neq M'$. In this section, we define various $2$-Selmer groups and list useful lemmas to be used in later sections. Let $S$ be a (finite) set of places of $K$ containing all primes above $2$, all primes where $E$ or $A$ has bad reduction, and all archimedean places. We enlarge $S$, if necessary so that $\Pic(\O_{K,S}) = 1$, where $\O_{K,S}$ denotes the ring of $S$-integers of $K$. We continue to assume that $C$ is an elliptic curve over $K$. Let $\l$ denote a place of $K$.

\begin{defn}
We define the restriction map
$$
\res_\l: H^1(K, C[2]) \to H^1(K_\l, C[2])
$$
as the restriction map of group cohomology. 
\end{defn}
Although the map $\res_\l$ depends on $C$, we suppress it from the notation. Which elliptic curve we take for the restriction map will be always clear from the context.

\begin{defn}
\label{localcondition}
For $\chi \in \Xset(K_v)$, we define
$$
\beta_{C,v}(\chi) := \mathrm{Im}\big( C^{\chi}(K_v)/2 C^{\chi}(K_v) \to H^1(K_v, C^{\chi}[2]) \cong H^1(K_v, C[2])\big),
$$
where the first map is the Kummer map. Define
$$
h_{C,v}(\chi):= \dimtwo\big(\beta_{C,v}(1_v)/(\beta_{C,v}(1_v) \cap \beta_{C,v}(\chi))\big),
$$
where $1_v \in \Xset(K_v)$ denotes the trivial homomorphism. 
\end{defn}

\begin{defn}
For $i=0,1,2,$ Define
\begin{align*}
\cP_{E, i} &:=\; \{\l : \text{$\l \notin S$
   and $\dim_{\Ftwo}(E(K_\l)[2]) = i$}\} \text{ and} \\
\cP_{A, i} &:=\; \{\l : \text{$\l \notin S$
   and $\dim_{\Ftwo}(A(K_\l)[2]) = i$}\}.\quad 
\end{align*}
Define $\cP_0 := \cP_{E,0} \cap \cP_{A,0}$. 
\end{defn}

\begin{rem}
\label{noname}
Recall that for any $\chi_\l \in \Xset(K_\l)$, there is a canonical isomorphism $E[2] \cong E^{\chi_\l}[2]$.  Hence, Lemma \ref{locdim2torsion} shows that if $\l \in \cP_{E,i}$, then 
$$\dimtwo(\beta_{E,\l}(1_\l)) = \dimtwo(\beta_{E, \l}(\chi_\l)) = i.$$  
The same is true if $E$ is replaced by $A$. 
\end{rem}

\begin{lem}
\label{cPFroborder}
Suppose that $\l \not\in S$. Then
\begin{enumerate}
\item
$\l \in \cP_{E, 0}$ if and only if $\Frob_\l|_M \in \Gal(M/K)$ has order $3$,
\item
$\l \in \cP_{E, 1}$ if and only if $\Frob_\l|_M \in \Gal(M/K)$ has order $2$,
\item
$\l \in \cP_{E, 2}$ if and only if $\Frob_\l|_M \in \Gal(M/K)$ is identity. 
\end{enumerate}
The same is true if $E, M$ are replaced by  $A, M'$, respectively. 
\end{lem}

\begin{proof}
Note that for every $\l \notin S$, the extension $M/K$ is unramified at $\l$. The lemma follows from Lemma \ref{locdim2torsion}. 
\end{proof}

\begin{rem}
\label{33}
Recall that $M \neq M'$. Suppose both $[M:K]$ and $[M':K]$ are divisible by $3$ (so, $\Gal(M/K), \Gal(M'/K)$ are $S_3$ or $\Z/3\Z$). Then, $[M\cap M':K]$ is not divisible by $3$ since otherwise $M$ and $M'$ would be the Galois closure of $M \cap M'$ over $K$ ($S_3$ has no normal subgroup of order $2$). Hence there exists $\sigma \in \Gal(MM'/K)$ such that
\begin{itemize}
\item
$\sigma|_{M\cap M'} = 1$
\item
$\sigma|_M$ has order $3$,
\item
$\sigma|_{M'}$ has order $3$.
\end{itemize}
By the Chebotarev density theorem there exist infinitely many primes $\l \not\in S$ such that $\Frob_\l|_{MM'} = \sigma$. Then by Lemma \ref{cPFroborder}, $\cP_0$ is an infinite set in this case.
\end{rem}

\begin{defn}
\label{selmer}
Let $\chi \in \Xset(K)$. The $2$-Selmer group $\Sel_2(C^\chi) \subset H^1(K,C[2])$ of $C^\chi$ (over $K$) is the (finite dimensional)
$\F_2$-vector space defined by the following exact sequence
$$
0 \too \Sel_2(C^\chi) \too H^1(K,C[2]) \too \dirsum{v} H^1(K_v,C[2])/\beta_{C, v}(\chi_v),
$$
where the rightmost map is the sum of the restriction maps, and $\chi_v$ is the restriction of $\chi$ to $G_{K_v}$. In particular, if $\chi$ is the trivial character, it is the classical $2$-Selmer group of $C$. 
\end{defn}

\begin{defn}
Define
$$r_2(C):=\dimtwo(\Sel_2(C))$$
for the sake of brevity.
\end{defn}

The following theorem due to Kramer gives a parity relation between $r_2(C)$ and $r_2(C^\chi)$ for $\chi \in \Xset(K)$. 

\begin{thm}[Kramer]
\label{parity}
Let $\chi \in \Xset(K)$. We have
$$
 r_2(C) -  r_2(C^\chi) \equiv \displaystyle{\sum_{v}}h_{C,v}(\chi_v) (\textrm{mod } 2),
$$
where $\chi_v$ is the restriction of $\chi$ to $G_{K_v}$.
\end{thm}

\begin{proof}
See for example \cite[Theorem 2.7]{MR} and \cite[Lemma 2.9]{MR}.
\end{proof}

\begin{lem}
 \label{localzero}
Let $\chi_v \in \Xset(K_v)$. Suppose that one of the following conditions holds:
\begin{itemize}
\item
$\chi_v$ is trivial, 
\item
 $v \nmid \infty$, $C/K_v$ has good reduction, and $\chi_v$ is unramified, 
\item
$v\nmid 2$ is a prime of $K$ and $C(K_v)[2] = 0.$
\end{itemize}
Then $\beta_{C,v}(1_v) = \beta_{C,v}(\chi_v)$, i.e., $h_{C,v}(\chi_v) = 0.$
 \end{lem}
\begin{proof}
The first case is easy to see. The third one follows from Lemma \ref{locdim2torsion}. Let $L = \overline{K}_v^{\ker(\chi_v)}$. In the second case, \cite[Corollary 4.4]{norm}) shows that $\N_{L/K_v}(C(L)) = C(K_v)$, where $\N_{L/K_v}$ is the norm map from $C(L)$ to $C(K_v)$.  Thus, by Lemma \cite[Lemma 2.9]{MR}, the result follows.
\end{proof}

\begin{rem}
For a given elliptic curve $C$ over $K$ and $\chi \in \Xset(K)$, there are only finitely many places, where either $C$ has bad reduction or $\chi$ is ramified. Thus, by Lemma \ref{localzero}, the sum in Theorem \ref{parity} is indeed a finite sum. 
\end{rem}

\begin{lem}
\label{ramhv}
Let $\l \nmid 2$ be a prime of $K$ and suppose $C$ has good reduction at $\l$. Let $\chi_\l \in \Xset_{\ram}(K_\l)$. Then 
$$\beta_{C, \l}(1_\l) \cap \beta_{C, \l}(\chi_\l) = \{0\} \text{ and } h_{C, \l}(\chi_\l) = \dimtwo(C(K_\l)[2]).$$ 
\end{lem}

\begin{proof}
This is \cite[Lemma 2.11]{MR}.
\end{proof}

We define strict, relaxed, and locally twisted $2$-Selmer groups as follows.

\begin{defn}
\label{variousselmer}
Define the strict 2-Selmer group at $\l$ 
$$
\Sel_{2, \l}(C) := \{x \in \Sel_{2}(C):\mathrm{res}_{\l}(x) = 0  \}.
$$
Define the relaxed 2-Selmer group at $\l$ 
$$
\Sel_2^{\l}(C) := \{x \in H^1(K, C[2]): \mathrm{res}_{v}(x) \in \beta_{C,v}(1_v) \text{ if } v  \neq \l \}.
$$
For $\psi_\l \in \Xset(K_\l)$, define
\begin{align*}
\Sel_{2}(C, \psi_\l) := \{x \in H^1(K, C[2]): \text{ }  & \mathrm{res}_{\l}(x) \in \beta_{C,\l}(\psi_\l), \text{ and } \\
& \mathrm{res}_{v}(x) \in \beta_{C, v}(1_v) \text{ if } v \neq \l\}.
\end{align*}
\end{defn}

\begin{thm}
\label{ptd}
The images of right hand restriction maps of the following exact sequences are orthogonal complements with respect to the pairing given by the pairing \eqref{tld} at $\l$
$$
\xymatrix@R=3pt@C=20pt{
0 \ar[r] & \Sel_2(C) \ar[r] & \Sel_2^\l(C) \ar[r]
    &  H^1(K_\l,C[2])/\beta_{C, \l}(1_\l),  \\
0 \ar[r] & \Sel_{2, \l}(C) \ar[r] & \Sel_2(C) \ar[r] & \beta_{C,\l}(1_\l).
}
$$
In particular, 
$$
\dimtwo(\Sel_2^\l(C)) - \dimtwo(\Sel_{2, \l}(C)) = \dimtwo(\beta_{C,\l}(1_\l))
= \frac{1}{2}\dimtwo(H^1(K_\l, C[2])).
$$
\end{thm}

\begin{proof}
The theorem follows from the Global Poitou-Tate Duality. For example, see \cite[Theorem 2.3.4]{kolyvagin}.
\end{proof}


\section{Metabolic spaces and Lagrangian subspaces}
In this section, we define Metabolic spaces, Lagrangian (maximal isotropic) subspaces, and canonical quadratic forms induced by the Heisenberg groups. We closely follow \cite{KMR}. For general theory, we refer the reader to section 2 and section 4 of \cite{poonenrain}. The main goal of this section is to prove Proposition \ref{crucial}, which will play a crucial role to prove Theorem \ref{case1}, Theorem \ref{case2}, and Theorem \ref{case3}. We continue to assume that $C$ is an elliptic curve over $K$. 

Let $V$ be a finite dimensional $\Ftwo$-vector space. 
\begin{defn}
\label{metabolic}
A {\em quadratic form} on $V$ is a function $q : V \to \Ftwo$ such that
\begin{itemize}
\item
$q(av) = a^2 q(v)$ for every $a \in \Ftwo$ and $v \in V$, \text{ and }
\item
the map $(v,w)_q := q(v+w)-q(v)-q(w)$ is a bilinear form.
\end{itemize}
We call $X$ a {\em Lagrangian subspace} or {\em maximal isotropic subspace} of $V$ if 
\begin{enumerate}
\item
$q(X) = 0$, and 
\item
$X$ = $X^{\perp}$ in the induced bilinear form.
\end{enumerate}
  A {\em metabolic space} $(V, q)$ is a vector space such that $( , )_q$ is nondegenerate and $V$ contains a Lagrangian subspace.
\end{defn}

\begin{defn}
Let $L$ be either $K$ or $K_v$ for a place $v$ of $K$. Define the Heisenberg group of $C$ over $L$
$$
\H_{C, L} := \{(f,P) \in \overline{L}(C) \times C[2] : \text{ the divisor of $f$ is $2[P] - 2[O]$},
$$
where $\overline{L}(C)$ is the function field of $C$ over $\overline{L}$, and $O$ is the trivial element of $C[2]$. The group law is defined by
$$
(f, P) \times (g, Q) := (\tau_Q^{\ast}(f)g, P+ Q), 
$$
where $\tau_Q$ is translation by $Q$ on $C$. 
\end{defn}

We have an exact sequence
\begin{equation}
\label{shortexact}
1 \to \overline{L}^\times \to \H_{C, L} \to C[2] \to 0,
\end{equation}
where the middle maps are defined by sending $l$ to $(l,O)$, and by taking $(f, P)$ to $P$, respectively. Let
$$
q_{\H_{C,L}} : H^1(L, C[2]) \to H^2(L, \overline{L}^\times) 
$$
be the connecting homomorphism of the long exact sequence of (non-abelian) Galois cohomology groups induced by \eqref{shortexact}.  By the construction, $q_{\H_{C,L}}$ is functorial with respect to base extension. 

\begin{defn}
Define
$$
q_{_{C,v}} : H^1(K_v, C[2]) \to H^2(K_v, \overline{K}_v^\times) \subset \Q/\Z
$$
be the composition of $q_{\H_{C,K_v}}$ and the invariant map $\mathrm{inv}_v : H^2(K_v, \overline{K}_v^\times) \to \Q/\Z$.
\end{defn}

Recall that for a quadratic form $q$, we associate a bilinear form $( , )_q$ (Definition \ref{metabolic}).

\begin{thm}
\label{listthm}
Suppose that $\chi\in \Xset(K_v)$. We have
\begin{enumerate}
\item
The bilinear form on $H^1(K_v, C[2])$ associated to $q_{_{C,v}}$ is exactly the pairing \eqref{tld} in Theorem \ref{tlp},
\item
$\beta_{C,v}(1_v)$ is a Lagrangian subspace of $(H^1(K_v, C[2]), q_{_{C,v}})$,
\item
$(H^1(K_v, C[2]), q_{_{C,v}})$ is a metabolic space.
\item
The canonical isomorphism $C[2] \cong C^\chi[2]$ identifies $q_{_{C,v}}$ and $q_{_{C^\chi,v}}$.
\item
$\beta_{C,v}(\chi)$ is a Lagrangian subspace of $(H^1(K_v, C[2]), q_{_{C,v}})$.
\end{enumerate}
\end{thm}

\begin{proof}
The assertion (i) is \cite[Corollary 4.7]{poonenrain}. The assertion (ii) follows from \cite[Proposition 4.9]{poonenrain}. (iii) is an easy consequence of (i) and (ii). The assertion (iv) is proved in \cite[Lemma 5.2]{KMR} and (v) follows from (ii) and (iv). 
\end{proof}

\begin{lem}
\label{brauer}
Suppose that $x \in H^1(K, C[2])$. Then
$$
\sum_{v} q_{_{C,v}}(\mathrm{res}_v(x)) = 0.
$$
\end{lem}

\begin{proof}
We have an exact sequence (see \cite[Theorem 8.1.17]{cohomology} for reference)
$$
\xymatrix{
0 \ar[r] & \mathrm{Br}(K) \ar[r] & \bigoplus_v \mathrm{Br}(K_v) \ar^-{\oplus \mathrm{inv}_v}[r] & \Q/\Z \ar[r] & 0.
}
$$
The lemma follows from the functoriality.
\end{proof}

\begin{lem}
\label{numberoflagrangian}
Suppose that $(V, q)$ is a metabolic space such that $\dimtwo(V) = 2n$. Then for
a given Lagrangian subspace $X$ of $V$, there are exactly $2^{n(n-1)/2}$ Lagrangian subspaces that intersect $X$ trivially; i.e.,
$$
|\{ Y: Y\text{ is a Lagrangian subspace such that }Y \cap X = \{0\} \}| = 2^{n(n-1)/2}.
$$
\end{lem}

\begin{proof}
This is immediate from Proposition 2.6 (b),(c), and (e) in \cite{poonenrain}.
\end{proof}

\begin{prop}
\label{reslag}
Let $\l$ be a prime of $K$. Then $\res_\l(\Sel_2^\l(C)) \subset H^1(K_\l, C[2])$ is a Lagrangian subspace of $(H^1(K_\l, C[2]), q_{_{C,\l}})$.
\end{prop}

\begin{proof}
Theorem \ref{listthm}(ii) shows that for $v \neq \l$, we have $q_{_{C,v}}(\res_v(x)) = 0$. It follows that $q_{_{C, \l}}(\res_\l(x)) = 0$ by Lemma \ref{brauer}. Therefore, $\res_\l(\Sel_2^\l(C))$ is contained in its orthogonal complement in the pairing $( , )_{q_{_{C, \l}}}$ of Definition \ref{metabolic}, that is \eqref{tld} by Theorem \ref{listthm}(i). Since $\dimtwo(\res_\l(\Sel_2^\l(C))) = \frac{1}{2} \dimtwo(H^1(K_\l, C[2]))$ by Theorem \ref{ptd}, $\res_\l(\Sel_2^\l(C))$ is indeed equal to its orthogonal complement. 
\end{proof}

\begin{lem}
\label{2inftor}
Let $\l \nmid 2$ be a prime of $K$, where $C$ has good reduction. Let $\eta \in \Xset_\ram(K_\l)$. Then 
$$C^\eta(K_\l)[2^\infty] = C^\eta(K_\l)[2],$$ 
where  $C$ is naturally regarded as an elliptic curve over $K_\l$, and $C^\eta$ is the quadratic twists of $C$ by $\eta$.
\end{lem}

\begin{proof}
Let $K_\l^\ur$ denote the maximal unramified extension of $K_\l$. We show 
$$C^\eta(K_\l^\ur)[2^\infty] = C^\eta(K_\l^\ur)[2],$$ 
then the result follows by taking the submodules of $\Gal(K_\l^\ur/K_\l)$-invariant elements on both sides. Under our assumption that $C$ has good reduction at $\l \nmid 2$, it is well-known that $C(K_\l^\ur)[2^\infty] = C[2^\infty]$. Thus, if $P \in C^\eta(K_\l^\ur)[2^\infty]$, then $\eta(\sigma)(P) = P$ for any $\sigma \in G_{K_\l^\ur}$ by the definition of quadratic twists. It follows that $2P = 0$, so $P \in C^\eta(K_\l^\ur)[2]$, where the assertion follows.
\end{proof}

\begin{lem}
\label{diff}
Let $\l \nmid 2$ be a prime of $K$, where $C$ has good reduction. Suppose that $C(K_\l)[2] = C[2]$. Let $\Xset_\ram(K_\l) = \{ \eta_1, \eta_2 \}$. Then $\beta_{C,\l}(1_\l), \beta_{C, \l}(\eta_1)$, and $\beta_{C, \l}(\eta_2)$ are Lagrangian subspaces of the metabolic space $(H^1(K_\l, C[2]), q_{_{C,\l}})$, whose pairwise intersections are trivial, i.e.,
$$
\beta_{C,\l}(1_\l) \cap \beta_{C, \l}(\eta_1) = \beta_{C,\l}(1_\l) \cap \beta_{C, \l}(\eta_2)  = \beta_{C,\l}(\eta_1) \cap \beta_{C, \l}(\eta_2) = \{0\}.
$$ 
\end{lem}

\begin{proof}
By Theorem \ref{listthm}(v), we have that $\beta_{C,\l}(1_\l), \beta_{C, \l}(\eta_1)$, and $\beta_{C, \l}(\eta_2)$ are Lagrangian subspaces. By Lemma \ref{ramhv}, it only remains to prove that 
$$\beta_{C,\l}(\eta_1) \cap \beta_{C, \l}(\eta_2) = \{0\}.$$ 
Let $F_\l$ be the unramified quadratic extension of $K_\l$. Then by \cite[Lemma 2.16]{mj}, it is enough to show that the natural injection $C^{\eta_1}(K_\l) \hookrightarrow C^{\eta_1}(F_\l)$ induces
\begin{equation}
\label{wts}
C^{\eta_1}(K_\l)/2C^{\eta_1}(K_\l) \cong C^{\eta_1}(F_\l)/2C^{\eta_1}(F_\l).
\end{equation}
By the isomorphism in the proof of Lemma \ref{locdim2torsion} and Lemma \ref{2inftor}, it follows that
$$
C^{\eta_1}(K_\l)/2C^{\eta_1}(K_\l) \cong C^{\eta_1}(K_\l)[2^\infty]/2C^{\eta_1}(K_\l)[2^\infty] =  C^{\eta_1}(K_\l)[2] = C^{\eta_1}[2],
$$
where we get the last equality by the assumption $C(K_\l)[2] = C[2]$ and the canonical isomorphism $C[2] \cong C^{\eta_1}[2]$. Similary,
$$
C^{\eta_1}(F_\l)/2C^{\eta_1}(F_\l) \cong C^{\eta_1}[2],
$$ 
whence \eqref{wts}.
\end{proof}

\begin{prop}
\label{crucial}
Suppose that $\l\nmid 2$. Suppose that $C$ has good reduction at $\l$ and $C(K_\l)[2] = C[2]$. Then
\begin{enumerate}
\item
If $\res_\l(\Sel_2(C)) = 0$, then there exists $\psi_\l \in \Xset_\ram(K_\l)$ so that $\res_\l(\Sel_2^\l(C)) = \beta_{C,\l}(\psi_\l)$, 
\item
If $\res_\l(\Sel_2(C)) = 0$ and $\res_\l(\Sel_2(C, \psi_\l)) \neq 0$ for some $\psi_\l \in \Xset_\ram(K_\l)$, then $\res_\l(\Sel_2^\l(C)) = \beta_{C, \l}(\psi_\l)$, so $\Sel_2^\l(C) = \Sel_2(C, \psi_\l)$. 
\item
If there is $s \in \Sel_2(C)$ such that $\res_\l(s) \neq 0$, then for any $\eta_\l \in \Xset_\ram(K_\l)$, we have $\dimtwo(\Sel_2(C, \eta_\l)) \le r_2(C)$,
\end{enumerate}
\end{prop}

\begin{proof}
Lemma \ref{locdim2torsion} and Theorem \ref{ptd} show
$$\dimtwo(\Sel_2^\l(C)) - \dimtwo(\Sel_{2, \l}(C)) = 2 \text{ and } \dimtwo(\res_\l(\Sel_2^\l(C))) = 2.$$ 
For (iii), we first note that
$$
\Sel_{2, \l}(C) \subset \Sel_2(C), \Sel_2(C, \eta_\l) \subset \Sel_{2}^\l(C).
$$
Let 
\begin{align*}
X & := \Sel_2(C)/\Sel_{2, \l}(C), \text{ and }\\
Y & := \Sel_2(C, \eta_\l)/ \Sel_{2, \l}(C),
\end{align*}
for convenience. Clearly $\dimtwo(X+Y) \le 2$ and Lemma \ref{ramhv} implies $X \cap Y = \{0\}$. The assumption of (iii) shows $\dimtwo(X) \ge 1$, so $\dimtwo(Y) \le 1$, where (iii) follows. We now show (i) and (ii). Let $\Xset_\ram(K_\l)=\{\eta_1, \eta_2\}$. Since $\res_\l(\Sel_2(C)) = 0$, we have
$$
\beta_{C, \l}(1_\l) \cap \res_\l(\Sel_2^\l(C)) = \{ 0\}. 
$$
Thus, $\beta_{C,\l}(\eta_1), \beta_{C,\l}(\eta_2)$, and $\res_\l(\Sel_2^\l(C))$ are Lagrangian subspace of $(H^1(K_\l, C[2]), q_{_{C, \l}})$, whose intersections with $\beta_{C, \l}(1_\l)$ are trivial. However, Lemma \ref{numberoflagrangian} asserts that there are only $2$ such Lagrangian subspaces. Hence by Lemma \ref{diff}, we conclude that
$$ 
\res_\l(\Sel_2^\l(C)) = \beta_{C,\l}(\eta_1) \text{ or } \res_\l(\Sel_2^\l(C)) = \beta_{C,\l}(\eta_2).
$$
Now it is easy to see (i) and (ii) by Lemma \ref{diff} again. 
\end{proof}


\section{Case 1: $[M:K], [M':K]$ are divisible by $3$}
In this section, we assume that $\Gal(M/K) \cong S_3$ or $\Z/3\Z$, and $\Gal(M'/K) \cong S_3$ or $\Z/3\Z$. Recall that we suppose $M \neq M'$ in this paper. 
Let $\Delta_E, \Delta_A$ denote the discriminants of models of $E$, $A$, respectively. Let $S$ be a set of places of $K$ defined as in the beginning of section $2$. Enlarge $S$ if necessary, so that $\Delta_E, \Delta_A \in \O_{K,S}^\times$.  

\begin{lem}
\label{eveni}
If $i= 0$ or $2$, then for every $\l \in \cP_{E,i}$, we have $\Delta_E \in (\O_\l^\times)^2$. 
\end{lem}

\begin{proof}
It is easy to see $K(\sqrt{\Delta_E})$ is the only (possibly trivial) quadratic extension over $K$ in $M$. Then Lemma \ref{cPFroborder} shows that $\Frob_\l|_M \in \Gal(M/K)$ fixes $\sqrt{\Delta_E}$, hence $\sqrt{\Delta_E} \in K_\l^\times$. Since $\Delta_E \in \O_{K,S}^\times$, the result follows immediately.  
\end{proof}
We recall $\cP_0 = \cP_{E,0} \cap \cP_{A,0}$. 
\begin{lem}
\label{kernelofthemaptopzero}
Define $\A \subset K^\times/(K^\times)^2$ by
$$
\A := \ker(\O_{K,  S}^\times /(\O_{K,  S}^\times)^2 \to \prod_{\l \in \cP_0}\O_\l^\times/(\O_\l^\times)^2).
$$
Then $\A$ is generated by $\Delta_E$ and $\Delta_A$ (they are possibly the same or even trivial). 
\end{lem}

\begin{proof}
By Lemma \ref{eveni}, clearly $\Delta_E, \Delta_A \in \A$. Suppose $\alpha \in \O_{K,  S}^\times /(\O_{K,  S}^\times)^2$ is not generated by $\Delta_E,\Delta_A$. Then $MM'$ and $K(\sqrt{\alpha})$ are linearly independent over $K$, i.e., $MM'\cap K(\sqrt{\alpha}) = K.$ By Remark \ref{33}, there exists $\sigma \in \Gal(MM'K(\sqrt{\alpha})/K)$ such that 
\begin{itemize}
\item
$\sigma |_{M} \in \Gal(M/K)$ has order $3$,
\item
$\sigma |_{M'} \in \Gal(M'/K)$ has order $3$, and
\item
$\sigma(\sqrt{\alpha})= -\sqrt{\alpha}$. 
\end{itemize}
Then by the Chebotarev Density Theorem, there exist infinitely many primes $v$ of $K$ satisfying $\Frob_v|_{MM'K(\sqrt{\alpha})} = \sigma$. By Lemma \ref{cPFroborder}, $v \in \cP_0$, and the image of $\alpha$ in the natural map
$$
\O_{K,  S}^\times /(\O_{K,  S}^\times)^2 \to \O_v^\times/(\O_v^\times)^2
$$
is not trivial. Hence the Lemma follows.  
\end{proof}

\begin{prop}
\label{case1mainlemma}
For even integers $0 \le i,j \le 2$, the set $\cP_{E, i} \cap \cP_{A,j}$ is an infinite set. Suppose that $\ell \in \cP_{E,i} \cap \cP_{A,j}$ and $\psi_\ell \in \Xset(K_\ell)$. Then, there exists $\chi \in \Xset(K)$ such that
$$\Sel_2(E^\chi) = \Sel_2(E, \psi_\ell)\text{ } \text{ } \text{ and } \text{ } \text{ }\Sel_2(A^\chi) = \Sel_2(A, \psi_\ell).$$
\end{prop}

\begin{proof}
One can see that $\cP_{E,i} \cap \cP_{A,j}$ is an infinite set when $i,j$ are even by a similar argument to Remark \ref{33} using Lemma \ref{cPFroborder} suitably. The condition that $i,j$ are even numbers implies that $\sqrt{\Delta_E}, \sqrt{\Delta_A} \in K_{\ell}^\times$ by Lemma \ref{eveni}. Therefore, $\psi_\ell(\Delta_E) = \psi_\ell(\Delta_A) = 1$. Let $S(\ell) := S \cup \{\ell\}$. Recall that $\Pic(\O_{K, S}) = 0$, so $\Pic(\O_{K, S(\ell)}) = 0$. Thus global class field theory shows that
$$
\Xset(K) = \Hom(\iK/K^\times, \{\pm1\}) = \textstyle\Hom((\prod_{v \in  S(\ell)}K_v^\times \times
      \prod_{\l\notin S(\ell)}\O_\l^\times)/\O_{K, S(\ell)}^\times, \{\pm1\}).
$$
Let $\cP^K$ denote the set of all places of $K$. Let
\begin{align*}
Q &:= \cP^K - \{\cP_0 \cup  S(\ell)\}, \\
J &:= \O_{K, S(\ell)}^\times ,\\
G &:= \prod_{\q \in \cP_0}\O_\l^\times, \text{ and }\\
H &:= \prod_{\l \in Q }\O_\l^\times \times \prod_{v \in  S(\ell)}K_v^\times.
\end{align*}
Suppose the map
\begin{align*}
\Phi: \Xset(K) = \Hom((G\times H)/J, \{\pm1\}) \too & \Hom(H, \{\pm1\}) \\
& \cong \prod_{\l \in Q }\Hom(\O_\l^\times, \{\pm1\}) \times \prod_{v \in  S(\ell)}\Hom(K_v^\times, \{\pm1\})
\end{align*}
is induced by the natural map $H \to (G \times H)/J$. Then \cite[Lemma 6.6(i)]{KMR} and Lemma \ref{kernelofthemaptopzero} show that $\mathrm{Im}(\Phi)$ is exactly
$$\{h \in \Hom(H,\{\pm1\}): h(\Delta_E) = h(\Delta_A) = 1\}.$$ Put $f_\mu \in \Hom(K_\mu^\times, \{\pm1\})$ for $\mu \in S(\ell)$ and $g_\nu \in \Hom(\O_\nu^\times, \{\pm1\})$ for $ \nu \in Q$ such that
\begin{itemize}
\item
$f_\ell = \psi_\ell$,
\item
$g_{\l}$ is trivial for $\l \in Q$,
\item
$f_v = 1_v$ for $v \in S$.
\end{itemize}
Note that 
\begin{align*}
f_\ell(\Delta_E) \cdot \prod_{\l \in Q} g_\l(\Delta_E) \cdot \prod_{v \in S} f_v(\Delta_E)& = 1, \\
f_\ell(\Delta_A) \cdot \prod_{\l \in Q} g_\l(\Delta_A) \cdot \prod_{v \in S} f_v(\Delta_A) &= 1.
\end{align*}
  Therefore there is a global character $\chi \in \Xset(K)$ such that
\begin{itemize}
\item
$\chi_{\ell} = \psi_{\ell}$,
\item
$\chi_{\l}|_{\O_\l^{\times}} = 1_\l$ for $\l \in Q$,
\item
$\chi_{v} = 1_v$ for $v \in  S$
\end{itemize}
where $\chi_\ell, \chi_\l, \chi_v$ are the restrictions of $\chi$ to $K_{\ell}^\times, K_{\l}^\times, K_{v}^\times$ via the local reciprocity maps, respectively. In particular, if $\l \in Q$, then $\chi_\l$ is an unramified character of $G_{K_\l}$. Then by Lemma \ref{localzero}, 
$$\beta_{E,\p}(1_\p) = \beta_{E,\p}(\chi_\p) \text{ and }\beta_{A,\p}(1_\p) = \beta_{A,\p}(\chi_\p)$$ 
for all places $\p$ but $\ell$, where $\chi_\p$ denotes the restriction of $\chi$ to $G_{K_\p}$. Therefore the result follows. 
\end{proof}

\begin{thm}
\label{case1}
Suppose that $[M:K], [M':K]$ are divisible by $3$ and $M \neq M'$. Then there exist infinitely many $\chi_1, \chi_2 \in \Xset(K)$ such that 
\begin{enumerate}
\item
$r_2(E^{\chi_1}) = r_2(E) + 2$ and $r_2(A^{\chi_1}) = r_2(A)$,
\item
$r_2(A^{\chi_2}) = r_2(A) + 2$ and $r_2(E^{\chi_2}) = r_2(E)$.
\end{enumerate}
\end{thm}

\begin{proof}
We prove (i), and (ii) follows similarly. Let $\tilde{s}$ denote the image of $s \in \Sel_2(E)$ in the restriction map
$$
\Sel_2(E) \subset H^1(K, E[2]) \to \Hom(G_M, E[2]).
$$
Let $L$ be the fixed field of $\bigcap_{s \in \Sel_2(E)} \ker(\tilde{s})$, so $[L:M]$ is a $2$-power. Let $N$ denote the Galois closure of $L$ over $K$. Then $[N:K] = 2^a \cdot 3$ for some $a$ by Lemma \ref{galois}. Then it follows that $[N \cap M' : K]$ is not divisible by $3$ since otherwise it would mean $9 | [N:K]$ by Remark \ref{33}. Therefore by the Chebotarev density theorem, there are infinitely many primes $\l \notin S$ such that
\begin{itemize}
\item
$\Frob_\l|_N = 1$
\item
$\l \in \cP_{A,0}$, i.e., $\Frob_\l|_{M'}$ has order $3$.
\end{itemize}
In particular, since $\Frob_\l |_M =1$, we have $\l \in \cP_{E,2}$ (Lemma \ref{cPFroborder}). By our construction, $\res_\l(\Sel_2(E)) = 0$, so $\Sel_2(E) = \Sel_{2, \l}(E)$. By Proposition \ref{crucial}(i), there exists $\psi_\l \in \Xset_{\ram}(K_\l)$ such that $\beta_{E, \l}(\psi_\l) = \res_\l(\Sel_2^\l(E))$. Then Proposition \ref{case1mainlemma} shows the existence of $\chi_1 \in \Xset(K)$ such that $\Sel_2(E^{\chi_1}) = \Sel_2^\l(E)$ and $\Sel_2(A^{\chi_1}) = \Sel_2(A)$, where the latter can be proved by Lemma \ref{localzero} (the third condition). Hence (i) follows from Lemma \ref{locdim2torsion} and Theorem \ref{ptd}. 
\end{proof}

\section{Case 2 : $[M:K]= 1$ or $2$, and $[M':K] = 3$ or $6$}
In this section, we assume that $[M:K] = 1 \text{ or } 2$, and $[M':K] = 3 \text{ or } 6$. Let $S$ be the set of places of $K$ as defined in section $4$. 

\begin{lem}
\label{4tor}
Suppose $\l \nmid 2$ and $E[4] \subset E(K_\l)$. Then for any non-trivial $\chi_\l \in \Xset(K_\l)$, we have $E^{\chi_\l}(K_\l)[2^\infty] = E^{\chi_\l}[2]$.  Moreover, there exists a natural isomorphism
$$
E^{\chi_\l}[2] \cong E^{\chi_\l}(K_\l)/2E^{\chi_\l}(K_\l).
$$
\end{lem}

\begin{proof}
The first assertion follows from the definition of quadratic twist. Then the isomorphism in the proof of Lemma \ref{locdim2torsion} shows the second assertion.
\end{proof}

\begin{thm}
\label{case2}
Suppose that  $[M:K] = 1$ or $2$, and $[M':K] = 3$ or $6$. Then there exist infinitely many $\chi \in \Xset(K)$ such that $r_2(E^\chi) = r_2(E) + 2$ and $r_2(A^\chi) = r_2(A)$.   
\end{thm}

\begin{proof}
Let $\tilde{s}$ denote the image of $s \in \Sel_2(E/K)$ in the restriction map
$$
\Sel_2(E) \subset H^1(K,E[2]) \to \Hom(G_M, E[2]).
$$
Let $L$ be the fixed field of $\bigcap_{s \in \Sel_2(E)} \ker(\tilde{s})$, so $[L:K]$ is a $2$-power. Define $N$ to be the Galois closure of $LK(E[4])K(\sqrt{\O_{K,S}^\times})$ over $K$, where $\O_{K,S}^\times$ is the group of $S$-units ($N$ is a finite extension of $K$ by Dirichlet's unit theorem, see \cite[Lemma 4.1]{mj2} for example). Then $[N:K]$ is a $2$-power as well by Lemma \ref{galois} and Lemma \ref{boolean}. Therefore there exists $\sigma \in \Gal(NM'/K)$ such that 
\begin{itemize}
\item
$\sigma |_{N} \in \Gal(N/K)$ is trivial,
\item
$\sigma |_{M'} \in \Gal(M'/K)$ has order $3$.
\end{itemize}
By the Chebotarev density theorem, there exist infinitely many $\l \not\in S$ such that $\Frob_\l|_{NM'} = \sigma$.  Then Lemma \ref{cPFroborder} shows that
$$
\l \in \cP_{E, 2} \cap \cP_{A, 0}.
$$
Put $f_\mu \in \Hom(K_\mu^\times, \{\pm1\})$ for $\mu \in S$ and $g_\nu \in \Hom(\O_\nu^\times, \{\pm1\})$ for $ \nu \not\in S$ such that
\begin{itemize}
\item
$f_v = 1_v$ for $v \in S$,
\item
$g_{\l}$ is not trivial, and
\item
$g_{\p}$ is trivial for $\p \notin S \cup  \{\l\}$.
\end{itemize}
Since $K(\sqrt{O_{K,S}^\times}) \subset N$ and $\Frob_\l|_N = 1$, we have $g_{\l}(O_{K,S}^\times) = 1$. Therefore Lemma \ref{cft} shows that there exists $\chi \in \Xset(K)$ satisfying
\begin{itemize}
\item
$\chi_{v} = 1_v$ for $v \in S$,
\item
$\chi_{\l}$ is ramified,
\item
$\chi_{\p}$ is unramified for $\p \notin S \cup \{\l\}$,
\end{itemize}
where $\chi_{v}, \chi_{\l},\chi_{\p}$ are restrictions of $\chi$ to $K_{v}^\times, K_{\l}^\times, K_{\p}^\times$ via the local reciprocity maps, respectively. Note that 
$\Sel_2(E) = \Sel_{2,\l}(E)$ and $\Sel_2(E^\chi) = \Sel_{2}(E, \chi_\l)$ by the choice of $\chi$ and Lemma \ref{localzero}. Hence Theorem \ref{ptd} proves that
$$
0 \le r_2(E^\chi) - r_2(E) \le 2.
$$
By our choice of $\l$, we have $E[4] \subset E(K_\l)$. By Lemma \ref{4tor}, there is an isomorphism
\begin{equation}
\label{mj2lem}
E^{\chi_\l}[2] \cong E^{\chi_{\l}}(K_\l)/2E^{\chi_{\l}}(K_\l) \cong \beta_{E, \l}(\chi_\l).
\end{equation}
Let $P$ be a non-trivial $K$-rational $2$-torsion point of $E^\chi$. Define a composition 
$$
\phi: E^\chi(K) \to E^\chi(K)/2E^\chi(K) \to H^1(K, E^\chi[2])
$$
where the first map is the projection and the second map is given by the Kummer map. We have $0 \neq \phi(P) \in \Sel_2(E^\chi)$ because $P$ is not trivial in $E^\chi(K)/2E^\chi(K)$ (note that $E^\chi(K)[2^\infty] \subset E^\chi(K_\l)[2^\infty] = E^\chi(K_\l)[2]$ by Lemma \ref{4tor}). The isomorphism \eqref{mj2lem} implies that $\res_\l(\phi(P)) \neq 0$. Therefore Theorem \ref{ptd} and Proposition \ref{crucial}(ii) show $r_2(E^\chi) = r_2(E) + 2$. Since $\l \in \cP_{A,0}$ and $\Sel_2(A^\chi) = \Sel_2(A, \chi_\l)$, Lemma \ref{localzero} proves that
$$
\Sel_2(A) = \Sel_2(A^\chi), \text{ and } r_2(A^\chi) = r_2(A).
$$
\end{proof}


\section{Case 3 : $[M:K] = 1$ or $2$, and $[M':K] = 2$}
In this section, we assume that $[M:K]=1$ or $2$. We assume further that $[M':K] = 2$ and $M \neq M'$. Let $S$ be the finite set of places of $K$ as defined in previous sections. For $t \in \Sel_2(A)$, we denote the image in the natural restriction map
$$
\Sel_2(A) \subset H^1(K, A[2]) \to \Hom(G_{M'}, A[2])
$$
by $\tilde{t}$ and let $\underline{t}$ denote the image of $t$ in the restriction map
$$
\Sel_2(A) \subset H^1(K, A[2]) \to H^1(M, A[2]).
$$

Let $T$ be a (finite) set of places of $M$ containing all primes above $2$, all primes where $E$ has bad reduction, and all archimedean places. In addition, we assume that $\Pic(\O_{M, T}) = 1$, where $\O_{M, T}$ is the ring of $T$-integers of $M$.  In the following Lemma, we use notations $\Sel_2(E/K), \Sel_2(E/M)$ to specify the base fields $K, M$ on which the Selmer groups are defined.  

\begin{lem}
\label{kerkerker}
There is a composition of injective group homomorphisms
$$
\Sel_2(E/K) \hookrightarrow \Sel_2(E/M) \hookrightarrow \Hom(\Gal(M(\sqrt{\O_{M, T}^\times})/M), E[2]).
$$
\end{lem}

\begin{proof}
The first map is given by the restriction map. The injectivity can be checked by \cite[Lemma 4.8]{mj2}. The second injection is given in the proof of \cite[Proposition 5.5]{mj2}.
\end{proof}

\begin{prop}
\label{22twists}
Suppose that there exists an element $t \in \Sel_2(A)$ such that the fixed field of $\ker(\tilde{t})$ is not contained in $M(\sqrt{M^\times})$. Then there exist infinitely many quadratic character $\chi \in \Xset(K)$ such that $r_2(E^\chi) = r_2(E) + 2$ and $r_2(A^\chi) \le r_2(A)$.
\end{prop}

\begin{proof}
Let $F$ be the fixed field of $\ker(\tilde{t})$. We denote the Galois closure of $F$ over $K$ by $R$. Choose $N$ as in the proof of Theorem \ref{case2}. Then Lemma \ref{kerkerker}, Lemma \ref{boolean}, and Lemma \ref{galois} show that $N \subset M(\sqrt{M^\times})$. The assumption on $\ker(\tilde{t})$ shows that there exists $\sigma \in \Gal(NR/K)$ satisfying
 \begin{itemize}
 \item
 $\sigma |_{NM'} = 1$,
 \item
 $\sigma |_{F} \neq 1$.
  \end{itemize} 
By Chebotarev's density theorem, there exist infinitely many primes $\l \notin S$ of $K$ such that $\Frob_\l |_{NR} = \sigma$. As in the proof of Theorem \ref{case2},  there exists $\chi \in \Xset(K)$ satisfying
\begin{itemize}
\item
$\chi_{v} = 1_v$ for $v \in S$,
\item
$\chi_{\l}$ is ramified,
\item
$\chi_{\p}$ is unramified for $\p \notin S \cup \{\l\}$,
\end{itemize}
where $\chi_{v}, \chi_{\l},\chi_{\p}$ are restrictions of $\chi$ to $K_{v}^\times, K_{\l}^\times, K_{\p}^\times$ via the local reciprocity maps, respectively. Then $$\Sel_2(E^\chi) = \Sel_2(E, \chi_\l) \text{ and } \Sel_2(A^\chi) = \Sel_2(A, \chi_\l)$$ 
by Lemma \ref{localzero}. By our assumption on $t$, we have $\res_\l(t) \neq 0$, so by Proposition \ref{crucial}(iii),  $r_2(A^\chi) \le r_2(A)$. The equality $r_2(E^\chi) = r_2(E) + 2$ follows from the proof of Theorem \ref{case2}. 
\end{proof}

\begin{prop}
\label{plus22}
There exist infinitely many $\chi \in \Xset(K)$ satisfying
\begin{enumerate}
\item
$r_2(A^\chi) = r_2(A) +2$,
\item
$r_2(E^\chi) = r_2(E) + 2$,
\item
there exists $s \in \Sel_2(A^\chi)$ so that (the fixed field of $\ker(\tilde{s})) \not\subset M(\sqrt{M^\times}).$ 
\end{enumerate}
\end{prop}

\begin{proof}
Choose $N$ as in the proof of Theorem \ref{case2}. Choose $N'$ for $A$ as we construct $N$ by replacing the role of $E$ with that of $A$. By the Chebotarev density theorem, there exist infinitely many $\l \not\in S$ such that $\Frob_\l|_{NN'} = 1$. As in the proof of Theorem \ref{case2}, there exists $\chi \in \Xset(K)$ such that
\begin{itemize}
\item
$\chi_{v} = 1_v$ for $v \in S$,
\item
$\chi_{\l}$ is ramified,
\item
$\chi_{\p}$ is unramified for $\p \notin S \cup \{\l\}$,
\end{itemize}
where $\chi_{v}, \chi_{\l},\chi_{\p}$ are restrictions of $\chi$ to $K_{v}^\times, K_{\l}^\times, K_{\p}^\times$ via the local reciprocity maps, respectively. Again by the proof of Theorem \ref{case2}, one can check (i) and (ii). By Lemma \ref{4tor}, there is an isomorphism $A^\chi[2] \cong A^\chi(K_{\l})/2A^\chi(K_{\l})$. By Proposition \ref{crucial}(ii) and the proof of Theorem \ref{case2}, we see 
$$\res_{\l}(\Sel_2(A^\chi)) = A^\chi(K_{\l})/2A^\chi(K_{\l}),$$ 
where the latter is identified with its Kummer image in $H^1(K_\l, A^\chi[2]) = \Hom(G_{K_{\l}}, A^\chi[2])$. We choose $d \in K^\times$ so that, the fixed field of $\ker(\chi)$ is $K(\sqrt{d})$. We define the composition of the maps
$$
\Phi: A[2] \cong A^\chi[2] \cong A^\chi(K_{\l})/2A^\chi(K_{\l}) \to \Hom(G_{K_{\l}}, A^\chi[2]) \cong \Hom(G_{K_{\l}}, A[2]).
$$
One can check that the map $\Phi$ takes $P \in A[2]$ to the homomorphism that sends $ \sigma \in G_{K_\l}$ to $(\chi(\sigma) - 1)Q$, where $P = 2Q$. Let $s \in \Sel_2(A^\chi)$ be such that $\res_{\l}(s) = \Phi(P_2)$, where $P_2 \in A[2] - A(K)[2].$ Let $P_1$ denote the non-trivial $2$-torsion point of $A(K)$. Suppose that (iii) does not hold (so in particular $\tilde{s}(\tau^2) = 0$ for $\tau \in G_M$). Then for any $\tau \in G_M$ with $\tau|_{MM'} \neq 1$ (so $\tau(P_1) = P_1$ and $\tau(P_2) = P_1 + P_2$), we have 
$$0 = \tilde{s}(\tau^2) = \underline{s}(\tau^2) = \underline{s}(\tau) + \tau \underline{s}(\tau),$$ 
so $\underline{s}(\tau) =  P_1$ or $0$. Choose $\xi \in G_{K_{\l}} \subset G_{MM'} \subset G_M$ such that $\xi(\sqrt{d}) = -\sqrt{d}$, so $\tilde{s}(\xi) = P_2$ by our choice of $s \in \Sel_2(A^\chi)$. However we have
$$
\tilde{s}(\xi) = \underline{s}(\xi) = \underline{s}(\tau \cdot \tau^{-1}\xi) = \underline{s}(\tau) + \tau \underline{s}(\tau^{-1}\xi) = P_1 \text{ or } 0,  
$$
where $\tau \in G_M$ is choosen so that $\tau |_{MM'} \neq 1$ (so $\tau^{-1}\xi$ also satisfies the same property). We get a contradiction, therefore (iii) also holds. 
\end{proof}

By applying Proposition \ref{plus22} and Proposition \ref{22twists}, we finally prove the following. 
\begin{thm}
\label{case3}
Suppose that $[M:K] = 1$ or $2$,  $[M':K] = 2$, and $M \neq M'$. Then there exist infinitely many $\chi \in \Xset(K)$ such that $r_2(E^\chi) - r_2(A^\chi) \ge r_2(E) - r_2(A) + 2.$
\end{thm}

\begin{rem}
In the proof of Theorem \ref{case3}, using Proposition \ref{plus22} is crucial in order to apply Proposition \ref{22twists}. Without this step, it could be possible that the fixed field of $\ker(\tilde{t})$ is contained in $M(\sqrt{M^\times})$ for every $t \in \Sel_2(A)$.
\end{rem}


\section*{Acknowledgements}
The author is very grateful to Professor Karl Rubin for directing him to \cite{companion} and discussions.

\bibliographystyle{abbrv}
\bibliography{forthmjrefe}

\end{document}